\documentclass{birkjour}   	% use "amsart" instead of "article" for AMSLaTeX format
\usepackage{amsmath, amssymb}
\usepackage[shortlabels]{enumitem}
\usepackage{mathrsfs}
\usepackage[T1]{fontenc}
\usepackage{amsthm}
\usepackage{mathtools,breqn, stackengine}
\usepackage{geometry,tikz}
\newtheorem{theorem}{Theorem}

\newtheorem{remark}{Remark}
\newtheorem{lem}{Lemma}

\usepackage[symbol]{footmisc}

\everymath{\displaystyle}
\usepackage{xcolor}

	\title[Optimal data for a Reaction-Advection-Diffusion model]{The optimal initial datum  for a class of reaction-advection-diffusion equations}
	
	\author{Omar Abdul Halim}
	\address{Department of Mathematics and Statistics\\
		University of Northern British Columbia\\
		Prince George, BC, V2N 4Z9}
	\email{abdul@unbc.ca}
	\author{Mohammad El Smaily}
	\address{Department of Mathematics and Statistics\\
		University of Northern British Columbia\\
		Prince George, BC, V2N 4Z9}
	\email{smaily@unbc.ca}
	\thanks{Mohammad El Smaily acknowledges partial support from the Natural Sciences and Engineering Research Council of Canada through the NSERC Discovery Grant RGPIN-2017-04313.}
	\thanks{Corresponding Author: Mohammad El Smaily}
	\date{August 19, 2021}
\begin{document}	
	
\maketitle
	\begin{abstract}We consider a reaction-diffusion model with a drift term in a bounded domain. Given a time $T,$ we prove the existence and uniqueness of an initial datum that maximizes the total mass $\textstyle{\int_\Omega u(T,x)\mathrm{d}x}$ in the presence of an advection term. In a population dynamics context, this optimal initial datum can be understood as the best  distribution of the initial population that leads  to a maximal the total population at a prefixed time $T.$  We also compare the total masses at a time $T$ in two cases: depending on whether an advection term is present  in the medium or not. We prove that the presence of a large enough advection enhances the  total mass.	\end{abstract}	
\section{Introduction and main results}
In this paper, we study an optimization problem associated with a reaction-diffusion model that  involves an advection term. The model reads 
\begin{equation}\label{eq}
	\left\{\begin{array}{ll}
		\partial_tu-\sigma \Delta u-q\cdot\nabla_x u=f(u) & \text{in $(0,T)\times\Omega$},\vspace{7 pt}\\
		u(0,x)=u_0(x) & \text{in $\Omega,$}\vspace{7 pt}\\
		\frac{\partial u}{\partial\nu}(t,x)=0, & \text{for all $t\in(0,T)$ and   $x\in\partial\Omega$},
	\end{array}\right.
\end{equation}
where $\Omega$ is a bounded, connected and smooth domain of $\mathbb{R}^N,$ $\sigma >0$ is a positive constant (standing for the diffusivity of the medium) and $\nu(x)$ is the outward unit normal to $\partial \Omega$ at  a point $x.$

Model \eqref{eq} appears in studies related to population dynamics , chemical reactions (see \cite{Cantrell} and \cite{Murray} for e.g.) and mixing processes (see \cite{Constantine} and the references therein). The function $u=u(t,x)$ then stands for the evolving density at a time $t$ and a location $x\in\Omega.$ Our present work aims to answer natural questions related to the maximization of the total mass $\textstyle{\int_\Omega u(T,x)~\rm{d}x}$ at a given time $T>0.$ More precisely, at a prefixed  time $T>0,$   our goal is to maximize the functional \begin{equation}\label{define.IT}\mathcal{I}_T(u_0):=\int_\Omega u(T,x)\,\mathrm{d}x
\end{equation} 
among all possible initial data $u_0\in \mathcal{A}_m,$ where 
\begin{equation}\label{m}\mathcal{A}_m:=\left\{u_0\in \mathcal{A},~ \int_\Omega u_0=m\right\}, \quad \mathcal{A}:=\bigg\{u_0\in L^1(\Omega),~ 0\le u_0\le1\bigg\},
\end{equation} 
and $u(t,x)$ is the solution of \eqref{eq} with an initial datum $u(0,\cdot)=u_0(\cdot).$

In \eqref{eq}, when $u$ represents the density of a population, the quantity $\textstyle{\int_\Omega u(t,x)\,\mathrm{d}x}$ is the total population at a given time $t.$ In a population dynamics context, our first problem can then be phrased as follows:
\begin{quote}{\it Given an initial total population $m,$ we investigate whether there is  an optimal   distribution of individuals, at the initial time, that makes  the total population at a time $T$  maximal. Moreover, if such an optimal initial datum exists, we seek answers related to its uniqueness.}
\end{quote}
The answers to these questions are given in the first two theorems of this paper.

\paragraph{Hypotheses} Let us state the assumptions we need on the terms $q$ and $f$ appearing in \eqref{eq}. Again, 
  $\Omega$ is a bounded, connected and smooth domain of $\mathbb{R}^N.$ The reaction term  $f$ satisfies the following assumptions: 
\begin{equation}\label{on.f}
	\begin{cases}
		f\in C^1(\mathbb{R}),\vspace{7 pt}\\
		\exists M>0, \text{ such that } \sup_{s\in[0,1]}|f'(s)|\le M,\vspace{7 pt}\\
		f(0)=f(1)=0.
	\end{cases}
\end{equation}
The underlying advection $q(x)=(q_1(x),\dots,q_N(x))$ is a vector field that satisfies 
\begin{equation}\label{cq}
	q\in C^{1,\delta}(\overline\Omega) ~\text{ for some $\delta>0,$
	}\quad\text{ and }\quad  q\cdot\nu=0 \text{ on }\partial \Omega.\end{equation}

\paragraph{On the regularity of solutions}The fact that $0\leq u_0\leq 1$ and $\Omega$ is bounded imply that $u_0\in L^p$ for any $p\geq 1.$ Since $q\in C^{1,\delta}(\Omega),$ and $f\in C^1$ it follows that  the solution $u$ is $C^{2,\alpha}$ (for some $0<\alpha<1$) in the space variable, and $C^1$ in $t$ (see Chapter 4 -- Friedman \cite{Friedman} and Wang \cite{Wang}).

Under the above assumptions, problem \eqref{eq} has a unique solution $u(t,x),$ such that \begin{equation}\label{0<u<1}
	0\le u(t,x)\le1, \text{ for all }(t,x)\in[0,T]\times\Omega.
	\end{equation}We explain \eqref{0<u<1} as follows.  Since $0\le u_0(x)\le1$ and $f(0)=f(1)=0,$ it follows that $\bar{U}=1$ is a supersolution of \eqref{eq} and $\underline{U}=0$ is a subsolution of \eqref{eq}. An application of the maximum principle  allows us to conclude that $0\leq u\leq 1.$

	The fact that $0\le u(t,x)\le1, \text{ for all }(t,x)\in[0,T]\times\Omega$  implies that the operator $\mathcal{I}_T:\mathcal{A}\to\mathbb{R},$ defined by
\[ \mathcal{I}_T(u_0)=\int_\Omega u(T,x)\,\mathrm{d}x,\]
where $u$ is the solution of equation \eqref{eq} assumes finite values, whenever the argument $u_0$  belongs to the family $\mathcal{A}.$ Note that  $u_0$ plays the role of an initial value for \eqref{eq}, which leads to  the integrand of our functional $\mathcal{I}_T$.

  Our first result answers the question about the existence of a maximizer. 
  
\begin{theorem}[Existence of an optimal initial datum]\label{th1} Let $\Omega$ be a bounded domain and let $f$ and $q$ satisfy the assumptions stated in \eqref{on.f} and \eqref{cq} respectively. Then, there exists $\bar{u}_0\in \mathcal{A}_m$ such that 	\[\max_{u_0\in \mathcal{A}_m}\mathcal{I}_T(u_0)=\mathcal{I}_T(\bar{u}_0).\]
\end{theorem}
The following result confirms the uniqueness of the optimal initial datum. 
\begin{theorem}[Uniqueness of the optimal initial datum]\label{th2}
	Suppose that $f$ and $q$ satisfy  the assumptions \eqref{on.f} and \eqref{cq} respectively. Furthermore, assume that $f$ is strictly concave. Then, the optimal initial datum $\bar{u}_0,$ obtained in Theorem \ref{th1}, is unique.
\end{theorem}
\begin{remark} We highlight the importance of the assumption that $f$ is concave in order to obtain the uniqueness of  a maximizer we have in Theorem \ref{th1} above. Indeed, the recent result  by Mazari, Nadin and Marrero  \cite{nadin2} shows the existence of a non-unique  maximizer when $f$ is convex. Remark 2 of \cite{nadin2} mentions at least two maximizers, whenever $f$ is convex. 
\end{remark}
Theorems \ref{th1} and \ref{th2} answer the first question posed in the introduction. 
\paragraph{Influence of advection on the maximal total mass} Another natural question is 
\begin{quote}{\it whether the presence of an advection term influences the value of the optimal mass $\textstyle{\int_\Omega u(T,x)\,\mathrm{d}x}$ or not. Our second goal is to compare the optimal total mass in the case where an advection is present to the optimal total mass in the case where no advection is considered in the model. Does the addition  of an advection to the medium enhance the total mass at a time $T$?}
\end{quote}
 The above question is addressed in the following theorem. Under an additional assumption on the divergence of the advection field, the next theorem shows that the  total mass at time $T$ in model \eqref{eqA} is larger  than the maximal total mass in model \eqref{eqV}, even when the reaction-diffusion equation in \eqref{eqV} has an optimal initial datum. We will see in the proof that this is mainly because \eqref{eqA} has an advection term, while model \eqref{eqV} does not. 
%\begin{theorem}[Enhancement of the total mass by advection]\label{th3}
%	We assume that  $f$ and $a$ satisfy   \eqref{on.f} and \eqref{cq} respectively.  Moreover, we assume that  
%	 \begin{equation}\label{div.q}\alpha:=\max_{x\in \overline{\Omega}}\nabla\cdot q(x)<0
%	\end{equation}  
%	and 
%	\begin{equation}
%f(s)\geq 0,~~ \text{for all }~s\in[0,1].
%	\end{equation}
%	For $U^0\in \mathcal{A}_m,$ let 
%	$\mathcal{I}^A_T(U^0)=\int_\Omega U_A(T,x)\,\mathrm{d}x,$ where $U_A$ is the solution of 
%	\begin{equation}\label{eqA}
%		\left\{\begin{array}{ll}
%			\partial_t U_A-\sigma \Delta U_A-Aq(x)\cdot \nabla_x U_A=f(U_A) & \text{in }(0,T)\times\Omega\vspace{7 pt}\\
%			U_A(0,x)=U^0(x) & \text{in }\Omega, \vspace{7 pt}\\
%			\frac{\partial U_A}{\partial\nu}(t,x)=0 & t\in(0,T), ~x\in\partial\Omega.
%		\end{array}\right.
%	\end{equation}
%Let $\mathcal{I}_{T}(V_0)=\int_\Omega V(t,x)\,\mathrm{d}x,$ where $V$ is the solution of 
%	\begin{equation}\label{eqV}
%		\begin{cases}
%			\partial_t V-\sigma \Delta V=f(V) & \text{in }(0,T)\times\Omega, \vspace{7 pt}\\
%			V(0,x)=V_0(x) & \text{in }\Omega, \vspace{7 pt}\\
%			\frac{\partial V}{\partial\nu}(t,x)=0 &  t\in(0,T), ~x\in\partial\Omega.
%		\end{cases}
%	\end{equation}
%	
%	
%If \begin{equation}\label{threshold.advection}
%A\ge -\frac{M|\Omega|}{m\alpha}>0,
%\end{equation}we then have 
%\begin{equation}\label{advection.enhances}
%\max_{u_0\in\mathcal{A}_m}\mathcal{I}_T^A(u_0)\geq \max_{V_0\in\mathcal{A}_m}\mathcal{I}_T(V_0).
% \end{equation}
%\end{theorem}

\begin{theorem}[Enhancement of the total mass by advection]\label{th3}
	We assume that  $f$ and $q$ satisfy   \eqref{on.f} and \eqref{cq} respectively.  Moreover, we assume that  
	 \begin{equation}\label{div.q}\alpha:=\max_{x\in \overline{\Omega}}\nabla\cdot q(x)<0
	\end{equation}  
	and 
	\begin{equation}
f(s)\geq 0,~~ \text{for all }~s\in[0,1].
	\end{equation}
	For $U_A^0\in \mathcal{A}_m,$ let 
	$\mathcal{I}^A_T(U^0)=\int_\Omega U_A(T,x)\,\mathrm{d}x,$ where $U_A$ is the solution of 
	\begin{equation}\label{eqA}
		\left\{\begin{array}{ll}
			\partial_t U_A-\sigma \Delta U_A-Aq(x)\cdot \nabla_x U_A=f(U_A) & \text{in }(0,T)\times\Omega\vspace{7 pt}\\
			U_A(0,x)=U_0(x) & \text{in }\Omega, \vspace{7 pt}\\
			\frac{\partial U_A}{\partial\nu}(t,x)=0 ,& t\in(0,T), ~x\in\partial\Omega.
		\end{array}\right.
	\end{equation}
Let $\mathcal{I}_{T}(V_0)=\int_\Omega V(t,x)\,\mathrm{d}x,$ where $V$ is the solution of 
	\begin{equation}\label{eqV}
		\begin{cases}
			\partial_t V-\sigma \Delta V=f(V) & \text{in }(0,T)\times\Omega, \vspace{7 pt}\\
			V(0,x)=V_0(x) & \text{in }\Omega, \vspace{7 pt}\\
			\frac{\partial V}{\partial\nu}(t,x)=0, &  t\in(0,T), ~x\in\partial\Omega.
		\end{cases}
	\end{equation}

If \begin{equation}\label{threshold.advection}
A\ge -\frac{M|\Omega|}{m\alpha}>0,
\end{equation}we then have 
\begin{equation}\label{advection.enhances}
\inf_{u_0\in\mathcal{A}_m}\mathcal{I}_T^A(u_0)\geq \max_{V_0\in\mathcal{A}_m}\mathcal{I}_T(V_0).
 \end{equation}
\end{theorem}
Theorem \ref{th3} shows that a large enough advection  makes the total mass of \eqref{eqA} larger than the optimal mass of \eqref{eqV} that can be obtained with an optimal distribution of the initial population. 
\section*{Generalization to more heterogeneous settings}
The results above can be generalized to more heterogeneous settings, where the diffusivity and reaction is space/space-time dependent. More precisely, we can generalize the previous results to a model of the form
\begin{equation}\label{eqD}
	\left\{\begin{array}{l}
		\partial_tu-\nabla\cdot (D(x)\nabla u) -q\cdot\nabla_x u=f(t,x,u) \text{ in } (0,T)\times\Omega,\vspace{7 pt}\\
		u(0,x)=u_0(x)\quad   \text{ in }\Omega,\vspace{7 pt}\\
		\nu \cdot D(x)\nabla u(t,x)=0, \quad   t\in(0,T) \text{ and  }x\in\partial\Omega,
	\end{array}\right.
\end{equation}
where $D(x)=(D_{ij}(x))_{1\le i,\,j\le n}$ denotes a $C^2(\overline\Omega)$  matrix, such that
\begin{equation}\label{elliptic}
\exists \theta>0,\,\forall \xi\in\mathbb{R}^n,\quad \xi \cdot D(x)\xi\geq \theta|\xi|^2.\end{equation}
The advection  $q$ satisfies \eqref{cq}. The reaction term $f$  in the generalized model \eqref{eqD} is a function satisfying
\begin{equation}\label{on.ftxu}
	\begin{cases}
		f\in C^1((0,T)\times\overline\Omega\times\mathbb{R}),\vspace{7 pt}\\
		 \exists M>0, \text{ such that }\sup_{(t,x,u)\in(0,T)\times\Omega\times[0,1]}\left|\frac{\partial f}{\partial u}(t,x,u)\right|\le M,\vspace{7 pt}\\
		\forall(t,x)\in(0,T)\times\Omega,~~f(t,x,0)=f(t,x,1)=0.
	\end{cases}
\end{equation}

As above, we denote by 
\[\mathcal{I}_T(u_0):=\int_\Omega u(T,x)\mathrm{d} x.\]
The following results are the generalizations of Theorems \ref{th1}, \ref{th2} and \ref{th3} to model \eqref{eqD}.
\begin{theorem}[Existence of an optimal initial datum]\label{th4} Let $\Omega$ be a bounded smooth domain and assume that $q,$ $D(x)$ and $f$ satisfy the assumptions stated in \eqref{cq}, \eqref{elliptic}  and \eqref{on.ftxu}  respectively. Then, there exists $\bar{u}_0\in \mathcal{A}_m$ such that 	\[\max_{u_0\in \mathcal{A}_m}\mathcal{I}_T(u_0)=\mathcal{I}_T(\bar{u}_0).\]
\end{theorem}

\begin{theorem}[Uniqueness of the optimal initial datum]\label{th5}
	Suppose that $q,$ $D(x)$ and $f$ satisfy the assumptions stated in \eqref{cq}, \eqref{elliptic}  and \eqref{on.ftxu}  respectively. Furthermore, assume that $f$ is strictly concave. Then, the optimal initial datum $\bar{u}_0,$ obtained in Theorem \ref{th4}, is unique.
\end{theorem}

\begin{theorem}[Enhancement of the total mass by advection]\label{th6}
	We assume that   $q,$ $D(x)$ and $f$ satisfy the assumptions stated in \eqref{cq}, \eqref{elliptic}  and \eqref{on.ftxu}.  Moreover, we assume that  
	\begin{equation}\label{div.qD}\alpha:=\max_{x\in \overline{\Omega}}\nabla\cdot q(x)<0
	\end{equation}  
	and 
	\begin{equation}
		f(t,x,s)\geq 0,~~ \text{for all }~(t,x,s)\in(0,T)\times\Omega\times[0,1].
	\end{equation}
	For $U_A^0\in \mathcal{A}_m,$ let 
	$\mathcal{I}^A_T(U^0)=\int_\Omega U_A(T,x)\,\mathrm{d}x,$ where $U_A$ is the solution of 
	\begin{equation}\label{eqAD}
		\left\{\begin{array}{l}
			\partial_t U_A-\nabla\cdot (D(x)\nabla U_A)-Aq\cdot \nabla U_A=f(t,x,U_A)~ \text{in }(0,T)\times\Omega\vspace{7 pt}\\
			U_A(0,x)=U_0(x)\quad  \text{in }\Omega, \vspace{7 pt}\\
			\nu\cdot D(x)\nabla U_A(t,x)=0,\quad  t\in(0,T), ~x\in\partial\Omega.
		\end{array}\right.
	\end{equation}
	Let $\mathcal{I}_{T}(V_0)=\int_\Omega V(t,x)\,\mathrm{d}x,$ where $V$ is the solution of 
	\begin{equation}\label{eqVD}
		\begin{cases}
			\partial_t V-\nabla\cdot (D(x)\nabla V)=f(t,x,V) & \text{in }(0,T)\times\Omega, \vspace{7 pt}\\
			V(0,x)=V_0(x) & \text{in }\Omega, \vspace{7 pt}\\
			\nu\cdot D(x)\nabla V(t,x)=0 &  t\in(0,T), ~x\in\partial\Omega.
		\end{cases}
	\end{equation}

	If \begin{equation}\label{threshold.advection}
		A\ge -\frac{M|\Omega|}{m\alpha}>0,
	\end{equation}we then have 
	\begin{equation}\label{advection.enhances}
		\inf_{u_0\in\mathcal{A}_m}\mathcal{I}_T^A(u_0)\geq \max_{V_0\in\mathcal{A}_m}\mathcal{I}_T(V_0).
	\end{equation}
\end{theorem}

\paragraph{Prior works} The first question in our paper is addressed 
 in Nadin and Marrero \cite{nadin1} in a more particular setting:  for a reaction-diffusion model without an advection field. The results of Nadin {\it et al.} \cite{nadin1} prove the existence of an optimal initial datum in the case where $q\equiv0$ in \eqref{eq}.  They also announce the open question about the uniqueness of the maximizer. We answer the uniqueness question in Theorem \ref{th2} of this present work, whenever $f$ is a concave function, despite the presence of a drift term.

 Garnier, Hamel and Roques \cite{Garnier} analyze the role of the spatial distribution of the initial condition in reaction-diffusion models of biological invasion. In \cite{Garnier}, the authors investigate the detrimental effect of fragmentation while we  investigate optimal initial data in this present work.

Theorem \ref{th3} in our paper is another result in the study of reaction-diffusion equation, where the advection has an important influence on the qualitative properties of solutions. The earlier works \cite{ek1}, \cite{ek2} and \cite{Zlatos2} address the influence of advection on speeding up the propagation of traveling fronts that are solutions of a reaction-advection-diffusion equation in an unbounded spatial domain. In these works, it is proved that the speed of propagation of traveling front solutions behaves as a linear function of the amplitude $A$ that we place here in front of the advection term.

 We also mention the work \cite{Zlatos1} that gives a class of flows, in the role of advection, that turns out to be especially efficient in speeding up mixing in a diffusive medium.  The results of \cite{Eproceedings} also show that the presence of an advection term can play a role in preventing the speed of propagation of traveling wave solutions from being monotone with respect to the diffusivity in the medium. 
 
Lastly, in all of the works \cite{ek1}, \cite{ek2}, \cite{Eproceedings}, \cite{Zlatos1} and \cite{Zlatos2}, mentioned above, the advection is assumed to be a divergence free vector field. This is not the case in  Theorem \ref{th3} in our present work.

\section{ Proofs of Theorems \ref{th1}, \ref{th2} and \ref{th3}}
The following lemma will be used in proving Theorem \ref{th1} mainly.  

\begin{lem}\label{bounded}
	Under the assumptions \eqref{on.f} and \eqref{cq}, the solution $u=u(t,x)$ of \eqref{eq} satisfies the following properties:
	
		\begin{equation}\label{second} 
		u\in L^2\left(0,T;H^1(\Omega)\right)\cap L^\infty\left(0,T;L^2(\Omega)\right)\text { and }
		\end{equation}
		\begin{equation}\label{third}
		\partial_t u\in L^2\left(0,T;H^{-1}(\Omega)\right).
		\end{equation}
\end{lem}

\begin{proof}[{\bf Proof of Lemma \ref{bounded}}]

	To prove \eqref{second}, we first use \eqref{0<u<1}  to conclude that 
	\[\|u(t,\cdot)\|_{L^2(\Omega)}^2\le\int_\Omega 1\,\,\mathrm{d}x=|\Omega|.\]
	Hence, $u\in L^\infty\left(0,T;L^2(\Omega)\right)\cap L^2\left(0,T;L^2(\Omega)\right).$ Now we multiply \eqref{eq} by $u$ and integrate by parts to obtain 
	\[\begin{array}{l}\frac{1}{2}\int_0^T\partial_t\left(\int_\Omega u^2\,\mathrm{d}x\right)\mathrm{d}t+\sigma\int_0^T\int_\Omega\left|\nabla u\right|^2\,\mathrm{d}x\mathrm{d}t-\int_0^T\int_\Omega (q\cdot\nabla u)u\,\mathrm{d}x\mathrm{d}t\vspace{7 pt}\\=\int_0^T\int_\Omega uf(u)\,\mathrm{d}x\mathrm{d}t.
	\end{array}\]
On the other hand, 
	\[\int_\Omega (q\cdot\nabla u) \,u\,\mathrm{d}x=\frac{1}{2}\int_{\partial\Omega}u^2q\cdot\nu \,\mathrm{d}x-\frac{1}{2}\int_\Omega(\nabla\cdot q)u^2\,\mathrm{d}x=-\frac{1}{2}\int_\Omega (\nabla\cdot q) u^2\,\mathrm{d}x,\]
due to the no flux condition in \eqref{cq}.  Therefore,  
\begin{equation}\label{ineq}
	\begin{array}{l}
	\frac{1}{2}\int_\Omega u^2(T,x)\mathrm{d}x-\frac{1}{2}\int_\Omega u_0^2\,\mathrm{d}x+\sigma\int_0^T\int_\Omega\left|\nabla_xu\right|^2\,\mathrm{d}x\mathrm{d}t\vspace{5pt}\\+\frac{1}{2}\int_0^T\int_\Omega (\nabla\cdot q) u^2\,\mathrm{d}x\mathrm{d}t=\int_0^T\int_\Omega uf(u)\,\mathrm{d}x\mathrm{d}t.
	\end{array}\end{equation}
Choosing $B$ such that  $|\nabla\cdot q|\le B,$ and since $|f'(u)|\le M,$ we obtain 
	\begin{equation}\label{boundedness1}\begin{array}{l}
		\sigma \| \nabla_xu\|^2_{L^2\left(0,T;L^2(\Omega)\right)}
		\le \int_0^T\int_\Omega uf(u)\,\mathrm{d}x\,\,\mathrm{d}t+\frac{1}{2}\|u_0\|^2_{L^2(\Omega)}-\frac{1}{2}\|u(T,\cdot)\|^2_{L^2(\Omega)}\vspace{7 pt}\\
		+\frac{B}{2}\int_0^T\|u(t,\cdot)\|_{L^2(\Omega)}^2\,\mathrm{d}t\vspace{7 pt}\\
		\le  M\|u\|_{L^\infty\left(0,T;L^2(\Omega)\right)}+B\|u\|^2_{L^2\left(0,T;L^2(\Omega)\right)}+\|u_0\|^2_{L^2(\Omega)}.
	\end{array}\end{equation}
Therefore, $u\in L^2\left(0,T;H^1(\Omega)\right).$
	
	In order to prove \eqref{third}, let $v\in H^1(\Omega)$ such that $\|v\|_{H^1(\Omega)}\le1.$ We  multiply \eqref{eq} by $v$ and then integrate by parts to get 
	\begin{equation}\label{towardsH-1}\int_\Omega(\partial_tu)v\,\mathrm{d}x+\sigma\int_\Omega\nabla_xu\cdot\nabla_xv\,\mathrm{d}x-\int_\Omega (q \cdot\nabla u)v\,\mathrm{d}x=\int_\Omega f(u)v\,\mathrm{d}x.
	\end{equation}
	Applying Cauchy-Schwartz inequality on  the third term of \eqref{towardsH-1} yields \[\begin{array}{ll}
		\left|   \int_\Omega (q \cdot\nabla u)v\,\mathrm{d}x\right| \le \|v\|_{L^2(\Omega)} \left(\int_\Omega (q\cdot\nabla u)^2\,\mathrm{d}x\right)^{\frac{1}{2}}
	&\le \gamma\|v\|_{L^2(\Omega)} \|\nabla u\|_{L^2(\Omega)},
	\end{array}\]
	where $\|q\|_{L^\infty(\Omega)}\le\gamma.$
	We also apply Cauchy-Schwartz inequality to the second  term of \eqref{towardsH-1} and this yields
	\[\int_\Omega (\partial_t u)v\,\mathrm{d}x\le \int_\Omega f(u)v\,\mathrm{d}x+\sigma\|\nabla u\|_{L^2(\Omega)} \|\nabla v\|_{L^2(\Omega)}+\gamma\|\nabla u\|_{L^2(\Omega)}\|v\|_{L^2(\Omega)}.\]
The latter means that the $H^{-1}(\Omega)$-norm of $\partial_tu(t,\cdot)$ satisfies the following estimate \begin{equation}\label{bounded2}
\|\partial_tu(t,\cdot)\|_{H^{-1}(\Omega)}\le M\|u(t)\|_{L^2(\Omega)}+\sigma\|\nabla_xu\|_{L^2(\Omega)}+\gamma\|\nabla u\|_{L^2(\Omega)}.
\end{equation} From \eqref{second}, it follows that both $\|u(t,\cdot)\|_{L^2(\Omega)}$ and $\|\nabla_xu(t,\cdot)\|_{L^2(\Omega)}$ are in $L^2(0,T).$ Therefore, $\partial_tu\in L^2\left(0,T;H^{-1}(\Omega)\right).$
\end{proof}

With the above lemma, we can now prove Theorem \ref{th1}. 
\begin{proof}[{\bf Proof of Theorem \ref{th1}}]
	We first prove the existence of maximizing  element for $\mathcal{I}_T$ over the set $\mathcal{A}_m.$ From \eqref{0<u<1}, we conclude that $\mathcal{I}_T$ is bounded.  This guarantees the  existence of a supremum value in the set $\mathcal{I}_T(\mathcal{A}_m)$ (the range of $\mathcal{A}_m$ under the map $\mathcal{I}_T$).   We can then consider a maximizing sequence $U_0^n$ in the set $\mathcal{A}_m,$ such that \[\lim_{n\to\infty} \mathcal{I}_T(U_0^n)=\sup_{U_0\in \mathcal{A}_m}\mathcal{I}_T(U_0).\]
	As $U_0^n\in L^\infty(\Omega), $ there exists $\bar{U}_0\in L^\infty(\Omega)$ such that  $U_0^n$ converges to $\bar{U}_0$ in the weak $\star$ topology.  For $\varphi=1\in L^1(\Omega),$ we get \[\int_\Omega U_0^n\varphi=\int_\Omega U_0^n=m\to\int_\Omega \bar{U}_0\varphi=\int_\Omega\bar{U}_0,~\text{ as }n\rightarrow+\infty.\] Hence, $\int_\Omega\bar{U}_0=m$ and thus $\bar{U}_0\in \mathcal{A}_m.$
	
	Let  $ U^n$ be the weak solution of \eqref{eq}, where the initial condition  $U_0^n$ (see \cite{perthame} for a review of weak solutions to parabolic equations).  That is,  for all $\varphi\in C^\infty\left((0,T)\times\Omega\right),$ we have 
	\begin{equation}\label{weak}\begin{array}{l}
	\int_\Omega U^n(T,x)\varphi(T,x)\,\mathrm{d}x-\int_\Omega U^n(0,x)\varphi(0,x)\,\mathrm{d}x\vspace{7 pt}\\-\int_0^T\int_\Omega U^n(t,x)\partial_t\varphi(t,x)\,\mathrm{d}x\mathrm{d}t-\sigma\int_0^T\int_\Omega U^n(t,x)\Delta_x\varphi(t,x)\,\mathrm{d}x\mathrm{d}t\vspace{5pt}\\+\int_0^T \int_\Omega U^n(t,x)q(x)\cdot\nabla_x\varphi(t,x)\,\mathrm{d}x\mathrm{d}t\vspace{5pt}\\
	+\int_0^T \int_\Omega U^n(t,x)(\nabla_x\cdot q(x))\cdot\varphi(t,x)\,\mathrm{d}x\mathrm{d}t\vspace{5pt}\\
	=\int_0^T\int_\Omega f( U^n(t,x))\varphi(t,x)\,\mathrm{d}x\mathrm{d}t.
	\end{array}\end{equation}

		Since $0\leq U^n\leq 1$ and $\Omega$ is bounded, the sequence $\{U^n(T,\cdot)\}$ is then bounded in  $L^2(\Omega),$ there exists $\bar{u}\in L^2(\Omega),$ such that $U^n(T,\cdot)$ converges weakly to $\bar{u}$ in $L^2(\Omega).$ Thus,  
	\[\forall \varphi\in L^2(\Omega),\quad \int_\Omega U^n(T,x)\varphi(x)\,\mathrm{d}x\to\int_\Omega \bar{u}(x)\varphi(x)\mathrm{d}x.\] The same reasons ($0\leq U^n\leq 1$ and $\Omega$ is bounded) imply that $\{U^n\}_n$ is bounded in $L^\infty \left(0,T; L^2(\Omega)\right).$ Hence, there exists $U\in L^\infty\left(0,T; L^2(\Omega)\right)$ such that $U^n$ converges, in the weak $\star$ topology,  to $U$ in $L^\infty\left(0,T;L^2(\Omega)\right).$ That is,  \[\begin{array}{c}\forall \varphi\in L^\infty\left(0,T;L^2(\Omega)\right),\vspace{7 pt}\\ 
	\int_0^T \int_\Omega U^n(t,x)\varphi (t,x)\,\mathrm{d}x \,\mathrm{d}t\to\int_0^T\int_\Omega U(t,x)\varphi(t,x)\,\mathrm{d}x \,\mathrm{d}t \text{ as }n\rightarrow+\infty.\end{array}\]
	
	Also, from \eqref{boundedness1} and \eqref{bounded2}, it follows that the  sequence  $\{\partial _t U^n\}$ is uniformly bounded in  $L^2\left(0,T;H^{-1}(\Omega)\right)$ by a constant independent of $n.$ Thus, there exists $v\in L^2\left(0,T;H^{-1}(\Omega)\right)$ such that $\partial_t U^n$ converges, in the weak $\star$ sense, to $v\in L^2\left(0,T;H^{-1}(\Omega)\right).$ That is, for all $\varphi\in L^2\left(0,T;H^1(\Omega)\right)$ we have 
	\[\int_0^T\int_\Omega\partial_t U^n(t,x)\varphi (t,x)\,\mathrm{d}x \,\mathrm{d}t\rightarrow \int_0^T \int_\Omega v(t,x)\varphi(t,x)\,\mathrm{d}x \,\mathrm{d}t, \text{ as } n\rightarrow+\infty.\]  
We note that,  for all $\varphi\in C^\infty_c(0,T)\times\Omega,$ we have 
	\[\int_0^T\int_\Omega\partial_t U^n(t,x)\varphi (t,x)\,\mathrm{d}x \,\mathrm{d}t=-\int_0^T\int_\Omega U^n(t,x)\partial_t\varphi(t,x)\,\mathrm{d}x \,\mathrm{d}t.\]
	 Passing to the limit as $n\rightarrow+\infty$ leads to $\partial_t U=v.$
	For $\varphi\in H^1\left(0,T;L^2(\Omega)\right),$ such that $\varphi(0,x)=0$ for all $x\in \Omega,$ we have
	\[\begin{array}{l}
	\int_0^T\int_\Omega \partial _t U^n(t,x)\varphi (t)\,\mathrm{d}x \,\mathrm{d}t\vspace{7 pt}\\
	=\int_\Omega U^n(T,x)\varphi (T,x)\,\mathrm{d}x-
	\int_0^T\int_\Omega U^n(t,x)\partial_t\varphi(t,x)\,\mathrm{d}x \,\mathrm{d}t.
	\end{array}
	\] Passing to the limit as $n\rightarrow+\infty$ in the latter, we get
	\[\begin {array}{l}
	\int_0^T\int_\Omega v(t,x)\varphi(t,x)\,\mathrm{d}x \,\mathrm{d}t\vspace{5pt}\\ =\int_\Omega\bar{u}(x)\varphi(T,x)\,\mathrm{d}x-\int_0^T\int_\Omega U(t,x)\partial_t \varphi(t,x)\,\mathrm{d}x \,\mathrm{d}t\vspace{7 pt} \\
=\int_\Omega\bar{u}(x)\varphi(T,x)\,\mathrm{d}x-\int_\Omega U(T,x)\varphi(T,x)\,\mathrm{d}x
+\int_0^T\int_\Omega\partial_t U(t,x)\varphi(t,x)\,\mathrm{d}x \,\mathrm{d}t\vspace{5pt}\\
=\int_\Omega\bar{u}(x)\varphi(T,x)\,\mathrm{d}x-\int_\Omega U(T,x)\varphi(T,x)\,\mathrm{d}x
+\int_0^T\int_\Omega v(t,x)\varphi (t,x)\,\mathrm{d}x \mathrm{d}t,
\end{array}
\] 
as $\partial_t U=v.$ Thus, $\bar{u}(x)=U(T,x)$ a.e in $\Omega.$
For $\varphi\in H^1\left(0,T;L^2(\Omega)\right),$ such that $\varphi(T,x)=0$ in $\Omega,$ we have
\[\begin {array}{l}
\int_0^T\int_\Omega \partial _t U^n(t,x)\varphi (t,x)\,\mathrm{d}x \,\mathrm{d}t\vspace{7 pt}\\
=\int_0^T\int_\Omega \bar{U}_0(x)\varphi(0,x)\,\mathrm{d}x-\int_0^T\int_\Omega U^n(t,x)\partial_t\varphi(t,x)\,\mathrm{d}x \,\mathrm{d}t.
\end{array}\] 
Passing to the limit as $n\rightarrow+\infty,$ we obtain 
\[\begin {array}{l}
\int_0^T\int_\Omega v(t,x)\varphi(t,x)\,\mathrm{d}x \,\mathrm{d}t\vspace{7 pt} \\
=-\int_\Omega\bar{U}_0(x)\varphi(0,x)\,\mathrm{d}x-\int_0^T\int_\Omega U(t,x)\partial_t \varphi(t,x)\,\mathrm{d}x \,\mathrm{d}t\vspace{7 pt} \\
=-\int_\Omega\bar{U}_0(x)\varphi(0,x)\,\mathrm{d}x+\int_\Omega U(0,x)\varphi(0,x)\,\mathrm{d}x\vspace{7 pt}\\
+\int_0^T\int_\Omega\partial_t U(t,x),\varphi(t,x)\,\mathrm{d}x \,\mathrm{d}t\vspace{5pt}\\
=\int_\Omega\bar{U}_0(x)\varphi(T,x)\,\mathrm{d}x-\int_\Omega U(T,x),\varphi(T,x)\,\mathrm{d}x\vspace{7 pt}\\
+\int_0^T\int_\Omega v(t,x)\varphi (t,x)\,\mathrm{d}x \,\mathrm{d}t.
\end{array}
\]  Hence, $\bar{U}_0(x)=U(0,x).$
From Lemma \ref{bounded},  we know that \[U^n\in D:=\left\{u\in L^2\left([0,T];H^1(\Omega)\right),~ \partial_t u\in L^2\left([0,T];H^{-1}(\Omega)\right)\right\}\] and $\{U^n\} _n$ is bounded in $D.$ Thanks to  Aubin-Lions Lemma \cite{lions}, the set $D$ is compactly embedded in $L^2\left([0,T];L^2(\Omega)\right).$  Hence, there exists a Cauchy subsequence of $\{U^n\}_n$ in $L^2\left([0,T];L^2(\Omega)\right).$ Then, the sequence $\{U^n\}_n$  converges strongly to $U$  in $L^2\left([0,T];L^2(\Omega)\right).$ As $f$ is Lipschitz, it  follows that  \[\int_0^T\int_\Omega f(U^n(t,x))\varphi(t,x)\,\mathrm{d}x \,\mathrm{d}t\to \int_0^T \int_\Omega f(U(t,x))\varphi(t,x)\,\mathrm{d}x \,\mathrm{d}t.\]
Taking the limit in \eqref{weak}, we get 
\[\begin{array}{l} \int_\Omega U(T,x)\varphi(T,x)\,\mathrm{d}x-\int_\Omega U(0,x)\varphi(0,x)\,\mathrm{d}x -\int_0^T\int_\Omega U(t,x),\partial_t\varphi(t,x)\,\mathrm{d}x\vspace{7 pt}\\-\sigma\int_0^T\int_\Omega U(t,x)\Delta\varphi(t,x)\,\mathrm{d}x+\int_0^T \int_\Omega U(T,x)q(x) \cdot\nabla_x\varphi(t,x)\,\mathrm{d}x\vspace{7 pt}\\=\int_0^T\int_\Omega f(U(t,x))\varphi(t,x)\,\mathrm{d}x.\end{array}\] This implies that $U$ is a weak solution of \eqref{eq}. Taking $\varphi=1$ yields  
\[\begin{array}{l}
\sup_{\mathcal{A}_m}\mathcal{I}_T(u_0)=\lim_{n\to\infty}\mathcal{I}_T(U_0^n)=\lim_{n\to\infty}\int_\Omega U^n(T,x)\,\mathrm{d}x\vspace{5pt}\\
=\lim_{n\to\infty}\int_\Omega U^n(T,x)\varphi(x)\,\mathrm{d}x=\int_\Omega U(T,x)\,\mathrm{d}x=\mathcal{I}_T(\bar{U}_0).
\end{array}\] Therefore, $\bar{U}_0$ is a maximizing element of $\mathcal{I}_T$ in $\mathcal{A}_m.$
\end{proof}
Now, we move to the proof of Theorem \ref{th2}, which addresses the uniqueness of the optimal initial datum.
\begin{proof}[{\bf Proof of Theorem \ref{th2}}]
First, we will prove that $\mathcal{I}_T$ is strictly concave. Let $\lambda\in(0,1)$  and let $U$ be the solution of  \eqref{eq} with the initial condition $ \lambda U_1^0(x)+(1-\lambda) U_2^0(x)$ and let $v(t,x)=\lambda U_1(t,x)+(1-\lambda) U_2(t,x),$ where $U_1$ and  $U_2$ are the solutions of \eqref{eq} with initial condition $U_1^0(x)$ and  $U_2^0(x)$ respectively. As $f$ is strictly concave, we have
\[\partial_t v-\sigma\Delta v-q\cdot\nabla_xv=\lambda f(U_1)+(1-\lambda) f(U_2)< f(\lambda U_1+(1-\lambda) U_2)=f(v).\]
From the equation satisfied by $U$ and the inequality satisfied by $v,$ and since $f$ is of class $C^1,$ it follows that $w:=v-U$ satisfies
\begin{equation}\label{difference}
\partial_t w-\sigma\Delta w-q\cdot\nabla_xw< f(v(t,x))-f(U(t,x)):= b(t,x) w,
\end{equation}
where $b(t,x)$ is a bounded continuous function  obtained from the fact that $f$ is $C^1(\mathbb{R})$ (hence Lipschitz) and $0<v,U<1.$ Namely, 
\[b(t,x)=\begin{cases}
\frac{f(v(t,x))-f(U(t,x))}{v(t,x)-U(t,x)}, \quad \text{ if }\quad v(t,x)\neq U(t,x)\vspace{7pt}\\
f'(v(t,x)) \text{ if }\quad v(t,x)= U(t,x).
\end{cases}\]
Moreover,
\[w(0,x)=v(0,x)-U(0,x)=0\text{ for all }x\in \Omega,~\text{ and }\frac{\partial w}{\partial \nu}=0 \text{ on }(0,T)\times \partial \Omega.\] Owing to the strong parabolic comparison principle on the function $w$ (see Theorem 7 and the remark afterwords in \cite{Weinberger}, for example), we conclude that $w\leq 0$ in $[0,T]\times \Omega.$ That is, $v\leq U$ in $[0,T]\times \Omega.$

 Now, we claim that $v(T,x)<U(T,x)$  in  an open set  contained in $\Omega.$ Suppose to the contrary that $v(T,x)=U(T,x)$ in $\Omega.$ We have \[\partial_t v-\sigma\Delta v-q\cdot\nabla_xv-f(v)<\partial_t u-\sigma\Delta U-q\cdot\nabla U-f(U)=0,\]
and since $v(T,x)=U(T,x),$  we can reduce $f(U)$ and $f(v)$  for $t=T$ to get 
\[\begin{array}{l}
\partial_t v(T,x)-\sigma\Delta v(T,x)-q(x)\cdot\nabla_xv(T,x)\vspace{7 pt}\\<\partial_t U(T,x)-\sigma\Delta U(T,x)-q(x)\cdot\nabla U(T,x). 
\end{array}\]
Again, since  $v(T,x)=U(T,x),$  we have \[\Delta v(T,x)=\Delta U(T,x)\text{ and } \nabla_xv(T,x)=\nabla_x U(T,x).\] This implies that 
\[\partial_t v(T,x)<\partial_t U(T,x). \] Hence, $\partial_t(u-v)(T,x)>0$ which means that $(U-v)(T,x)$ is increasing at $t=T.$ Then, there exists $\varepsilon>0$ such that \[U(T-\varepsilon,x)-v(T-\varepsilon,x)<U(T,x)-v(T,x)=0.\]
This contradicts the fact that $v(t,x)\le U(t,x)$ and  proves our claim.
Therefore, \[\begin{array}{cl} \mathcal{I}_T(\lambda U_1^0+(1-\lambda) U_2^0)&=\mathcal{I}_T(\lambda U_1^0(x)+(1-\lambda) U_2^0(x))\vspace{5pt}\\
&=\int_\Omega U(t,x)\,\mathrm{d}x\vspace{5pt}>\int_\Omega v(T,x)\,\mathrm{d}x\vspace{5pt}\\&=\lambda \mathcal{I}_T(U_1^0)+(1-\lambda) \mathcal{I}_T(U_2^0).\end{array}\] This means that $\mathcal{I}_T$ is strictly concave and  guarantees the uniqueness of the maximal element.

In order to prove the uniqueness of the above maximal element, we suppose that there exists $V_1^0$ and $V_2^0$ in $\mathcal{A}_m,$ such that both are maximizers of $\mathcal{I}_T$ in $\mathcal{A}_m.$ Then, \[\mathcal{I}_T(V_1^0)=\mathcal{I}_T(V_2^0)=\sup_{u_0\in\mathcal{A}_m}\mathcal{I}_T(u_0).\] Let $\mu$ be such that $0<\mu<1.$ We then have $$V=\mu V_1^0+(1-\mu)V_2^0\in \mathcal{A}_m,$$ because $\int_\Omega V(x)\,\mathrm{d}x=m.$ Since $\mathcal{I}_T$ is strictly concave, it follows that \[\mathcal{I}_T(V)>\mu \mathcal{I}_T(V_1^0)+(1-\mu)\mathcal{I}_T(V_2^0)=\sup_{u_0\in\mathcal{A}_m}\mathcal{I}_T(u_0).\] However, this contradicts the fact that $V_1^0$ and $V_2^0$ are maximizers of $\mathcal{I}_T$ in $\mathcal{A}_m.$ Therefore, the maximizer of $\mathcal{I}_T$ is unique in the set $\mathcal{A}_m.$
\end{proof}
Now we prove Theorem \ref{th3} under the assumption \eqref{div.q} on the advection field $q.$
\begin{proof}[{\bf Proof of Theorem \ref{th3}.
}]
Let $U_0\in\mathcal{A}_m$ be an arbitrary initial datum for \eqref{eqA}.  This means that $\mathcal{I}_T^A(U_0)=\int_\Omega U_A(T,x)~\mathrm{d}x$ is not necessarily the largest value in the set $I_T^A(\mathcal{A}_m).$
We integrate equation  \eqref{eqA}, in both time and space, to get 
\begin{equation}\label{eq1}
\begin{array}{ll}
\int_\Omega U_A(T,x)\,\mathrm{d}x
=&\int_\Omega U_A(0,x)\,\mathrm{d}x +A\int_0^T\int_\Omega q(x)\cdot\nabla U_A(t,x)\,\mathrm{d}x\,\mathrm{d}t\vspace{7 pt}\\
&+\int_0^T\int_\Omega f(U_A(t,x))\,\mathrm{d}x\,\mathrm{d}t.
\end{array}
\end{equation}
An integration by parts of  the second term on the right hand side of \eqref{eqA} yields
\begin{equation}\label{drifted}
\begin{array}{cl}
\int_\Omega U_A(T,x)\,\mathrm{d}x=&\int_\Omega U_A(0,x)\,\mathrm{d}x -A\int_0^T\int_\Omega(\nabla\cdot q(x))U_A(t,x)\,\mathrm{d}x\,\mathrm{d}t \vspace{7 pt}\\
&+\int_0^T\int_\Omega f(U_A(t,x))\,\mathrm{d}x\,\mathrm{d}t.
\end{array}
\end{equation}
 We claim now that the map $\textstyle{t\mapsto \int_\Omega U_A(t,x)\,\mathrm{d}x}$ is increasing. In fact, if we integrate \eqref{eqA} over $(0,t)\times\Omega$, we then get 
\begin{equation}
\begin{array}{ll}
\int_\Omega U_A(t,x)\,\mathrm{d}x=&m -A\int_0^t\int_\Omega(\nabla\cdot q(x))U_A(s,x)\,\mathrm{d}x\,\mathrm{d}s\vspace{7 pt}\\
&+\int_0^t\int_\Omega f(U_A(s,x))\,\mathrm{d}x\,\mathrm{d}s.
\end{array}
\end{equation}
As $t$ increases, the second and third term on the right hand side  increase in $t$ because the integrand functions  are positive. Hence, our claim follows.

Now we integrate the first equation of \eqref{eqV}  in both variables $t$ and $x$ to get 
\begin{equation}\label{uT}
\begin{array}{lll}
\int_\Omega V(T,x)\,\mathrm{d}x&=&\int_\Omega V_0(x)\,\mathrm{d}x+\int_0^T\int_\Omega f(V)\,\mathrm{d}x\, \mathrm{d}t\vspace{7 pt}\\
&=&m+\int_0^T\int_\Omega f(V)\,\mathrm{d}x\mathrm{d}t.
\end{array}
\end{equation}
Let $$A\ge -\frac{M|\Omega|}{m\alpha}>0,$$ where $\alpha=\max_{x\in\overline{\Omega}}\nabla\cdot q(x)<0$ is the constant appearing in the assumption \eqref{div.q}.
Using the fact that the map $\textstyle{t\mapsto \int_\Omega U_A(t,x)\,\mathrm{d}x}$ is increasing,  we get  \[\int_\Omega U_A(t,x)\,\mathrm{d}x\geq \int_\Omega U_A(0,x)\,\mathrm{d}x, ~~\text{for any $t>0.$}\] 
We will use this as follows:  
\begin{align*}
\mathcal{I}_T^A(U_0)&=     \int_\Omega U_A(T,x)\,\mathrm{d}x\vspace{7 pt}\\
&=m -A\int_0^T\int_\Omega(\nabla\cdot q(x))U_A(t,x)\,\mathrm{d}x\,\mathrm{d}t+\int_0^T\int_\Omega f(U_A(t,x))\,\mathrm{d}x\,\mathrm{d}t.\vspace{7 pt}\\
&\ge m-A\alpha\int_0^T\int_\Omega U_A(t,x)\,\mathrm{d}x\,\mathrm{d}t\ge m-A\alpha\int_0^T\int_\Omega U_A(0,x)\,\mathrm{d}x\,\mathrm{d}t\vspace{7 pt}\\
&\ge m-A\alpha Tm\ge  m+\frac{M|\Omega|}{m\alpha}\alpha Tm= m+ M\int_0^T\int_\Omega\,\mathrm{d}x\,\mathrm{d}t.
\end{align*}
However,  the solution $V$ satisfies $0\leq V\leq 1$ (for the same reasons given in explaining \eqref{0<u<1}). This and the previous inequality then lead to 
\begin{equation}\label{lead.to.enhancement}
\begin{array}{rl}
\mathcal{I}_T^A(U_0)\ge& m+ M\int_0^T\int_\Omega V(t,x)\,\mathrm{d}x\,\mathrm{d}t \vspace{7 pt}\\
\ge& m+ \int_0^T\int_\Omega f(V(t,x))\,\mathrm{d}x\,\mathrm{d}t\vspace{7 pt}\\
=&\int_\Omega V(T,x)\,\mathrm{d}x=\mathcal{I}_T(V_0) ~\text{ (from \eqref{uT})}.
\end{array}
\end{equation}
The passage from the first line to the second in \eqref{lead.to.enhancement} is based on the Mean Value Theorem's application on the function $f,$ which satisfies $\|f'\|_\infty\leq M$ and $f(0)=0:$
\[\forall (t,x), ~f(V(t,x))=f(V(t,x))-f(0)=f'(C(t,x))V(t,x)\leq M V(t,x),\] for some constant $0<C(t,x)<V(t,x).$ 

Recalling that the initial datum $U_0$ was arbitrarily chosen from the set $\mathcal{A}_m,$ and choosing $V_0$ to be such that 
\[\mathcal{I}_T(V_0)=\max_{v_0\in\mathcal{A}_m}\mathcal{I}_T(v_0),\] we conclude that 
\[\inf_{u_0\in \mathcal{A}_m}\mathcal{I}_T^A(u_0)\ge \max_{v_0\in\mathcal{A}_m}\mathcal{I}_T(v_0)\] and this completes the proof of Theorem \ref{th3}.
\end{proof}

\begin{remark}\label{no.contra} We mention that the right hand side of \eqref{drifted} contains the amplitude $A$ of the advection term $q.$ Having the assumption $\nabla\cdot q \leq \alpha<0,$ one needs to make sure that the quantity \[-A\int_0^T\int_\Omega(\nabla\cdot q(x))U_A(t,x)\,\mathrm{d}x\,\mathrm{d}t\] does not diverge to $+\infty,$ when $A\rightarrow +\infty.$ Otherwise, one will have a contradiction with the fact that the left hand side is a bounded quantity {\rm(}bounded below by $0$ and above by $|\Omega|,$ as $0\leq U_A\leq 1${\rm)}. In fact, if $\{A_n\}_n$ is a sequence that goes to $+\infty,$ one can prove that the corresponding solutions $U_{A_n}$ converge (at least in the sense of distributions) to a function $w,$ known as a first integral of the drift $q$ {\rm (}see \cite{ek2}, for e.g.{\rm)}, characterized by \[q\cdot \nabla w=0\text{ a.e. in }\Omega.\] As a consequence, we have \[\begin{array}{ll}
\lim_{A\rightarrow+\infty}\int_\Omega(\nabla\cdot q(x))U_A(t,x)\,\mathrm{d}x&=-\lim_{A\rightarrow+\infty}\int_\Omega q(x) \cdot\nabla U_A(t,x)\,\mathrm{d}x\vspace{7 pt}\\
&=\int_\Omega q(x) \cdot\nabla w(x)\,\mathrm{d}x=0.
\end{array}\]
The latter guarantees that the right hand side remains bounded even when the amplitude $A$ is large. 
\end{remark}
\section{On the proofs of Theorems \ref{th4}, \ref{th5} and \ref{th6}}
The proof of Theorem \ref{th4} is similar to the proof of Theorem \ref{th1}. The main difference appears in Lemma \ref{bounded}. In what follows, we will highlight the differences that would appear in proving an analogue of Lemma \ref{bounded}, which leads to the results in Theorem \ref{th4}. In the setting of Theorem \ref{th4}, after multiplying  both sides of the reaction-advection-diffusiuon equation in \eqref{eqD}, and then integrating by parts, we get the term 
\[\begin{array}{l}
	-\int_0^T\int_\Omega \nabla_x \cdot (D(x)\nabla_x u(t,x))u(t,x)\mathrm{d}x\mathrm{d}t\vspace{7 pt}\\
	=\int_0^T\int_\Omega \nabla_x\cdot(D(x)\nabla_x u(t,x))\mathrm{d}x\mathrm{d}t
	-\int_0^T\int_{\partial \Omega} [\underbrace{\nu\cdot D(x)\nabla_x u(t,x)}_{=0}]u(t,x)~\mathrm{d}x\mathrm{d}t\vspace{7 pt}\\
	\geq\theta \|\nabla_xu\|_{L^2(0,T;L^2(\Omega))}^2.
	\end{array}\]
The latter allows us to conclude that 
\[\theta\|\nabla_xu\|^2_{L^2\left(0,T;L^2(\Omega)\right)}
\le  M\|u\|_{L^\infty\left(0,T;L^2(\Omega)\right)}+B\|u\|^2_{L^2\left(0,T;L^2(\Omega)\right)}+\|u_0\|^2_{L^2(\Omega)},\]
where $\theta$ is the coercivity constant assumed in \eqref{elliptic} on the diffusion matrix $D.$

In \eqref{towardsH-1}, after integrating the second term by parts, we obtain 
\[\begin{array}{ll} \int_\Omega \nabla_x u(t,x)D(x)\nabla_x v(t,x)\,\mathrm{d}x&\le \|D\|_\infty\int_\Omega \nabla_x u(t,x)\nabla_x v(t,x)\,\mathrm{d}x\vspace{7 pt}\\
	&\le\|D\|_\infty \|\nabla_x u\|_{L^2(\Omega)}\|\nabla_x v\|_{L^2(\Omega)},
	\end{array}\]
where $\|D\|_\infty:=\max_{1\leq i,j\leq n}\{\max_\Omega |D_{ij}(x)|\}.$  Hence,
	\[\int_\Omega (\partial_t u)v\,\mathrm{d}x\le \int_\Omega f(u)v\,\mathrm{d}x+\|D\|_\infty\|\nabla u\|_{L^2(\Omega)} \|\nabla v\|_{L^2(\Omega)}+\gamma\|\nabla u\|_{L^2(\Omega)}\|v\|_{L^2(\Omega)}.\] 
	In \eqref{weak}, we take the test function $\varphi$ to be in  $C^\infty\left((0,T)\times\Omega\right)$ and require that it has a  compact support in the space variable $x.$ The diffusion term will then give rise to
	\[\begin{array}{l}\int_0^T\int_\Omega \nabla\cdot (D(x)\nabla_xU^n(t,x))\varphi(t,x)\,\mathrm{d}x\,\mathrm{d}t\vspace{7 pt} \\=\int_0^T\int_\Omega U^n(t,x) \nabla\cdot (D(x)\nabla_x\varphi(t,x))\,\mathrm{d}x\,\mathrm{d}t\vspace{7 pt} \\
		-\int_0^T\int_{\partial\Omega} \nabla_xu(t,x)D(x)\nabla_x\varphi(t,x)\,\mathrm{d}x\,\mathrm{d}t\vspace{7 pt}\\
		=\int_0^T\int_\Omega U^n(t,x) \nabla\cdot (D(x)\nabla_x\varphi(t,x))\,\mathrm{d}x\,\mathrm{d}t,
		\end{array}\] as $\varphi$ is compactly supported in the $x$ variable. The rest of the proof  of Theorem \ref{th4} is the same as that of Theorem \ref{th1}. 
	
	We omit the proofs of Theorem \ref{th5} and Theorem \ref{th6} as they are similar to those of Theorem \ref{th2} and Theorem \ref{th3} respectively.

\end{document}